\documentclass[12pt]{amsart}
\usepackage[utf8]{inputenc}
\usepackage{hyperref}

\title[The monoid of monotone functions]{The monoid of monotone functions on a poset\\and quasi-arithmetic multiplicities for uniform matroids}
\author[Bruns]{Winfried Bruns}
\address{Institut f\"ur Mathematik\\Universit\"at Osnabr\"uck\\49074 Osnabr\"uck, Germany}
\email{wbruns@uos.de}
\author[García-Sánchez]{Pedro A. García-Sánchez}
\thanks{The second author is supported by the project MTM2017--84890--P, which is funded by Ministerio de Economía y Competitividad and Fondo Europeo de Desarrollo Regional FEDER, and by the Junta de Andalucía Grant Number FQM--343. The third author is supported by the project PRIN 2017YRA3LK, funded by MIUR, Italy.}
\address{Departamento de \'Algebra and IEMath-GR\\Universidad de Granada\\18071 Granada, Espa\~na}
\email{pedro@ugr.es}
\author[Moci]{Luca Moci}
\address{Dipartimento di Matematica, Piazza Porta S. Donato 5, 40126 Bologna, Italy}
\email{luca.moci2@unibo.it}
\date{\today}
\usepackage[margin=3cm]{geometry}

\usepackage[pdftex]{graphicx}

\usepackage{tikz}
\usetikzlibrary{arrows}
\usetikzlibrary{positioning}
\usepackage{enumerate}
\usepackage{todonotes}

\usepackage[pdf]{graphviz}
\usepackage{url}

\usepackage{fourier}

\newtheorem{theorem}{Theorem}
\newtheorem{lemma}[theorem]{Lemma}
\newtheorem{proposition}[theorem]{Proposition}
\newtheorem{corollary}[theorem]{Corollary}
\newtheorem{problem}[theorem]{Problem}

\theoremstyle{remark}
\newtheorem{example}[theorem]{Example}
\newtheorem{remark}[theorem]{Remark}
\theoremstyle{definition}

\newtheorem{conjecture}[theorem]{Conjecture}

\def\cI{\mathcal{I}}
\def\cJ{\mathcal{J}}
\def\cK{\mathcal{K}}
\def\rM{\mathrm{M}}
\def\rC{\mathrm{C}}
\def\rL{\mathrm{L}}
\def\rU{\mathrm{U}}
\def\rI{\mathrm{I}}
\def\ZZ{{\mathbb Z}}
\def\NN{{\mathbb N}}
\def\QQ{{\mathbb Q}}

\DeclareMathOperator{\supp}{supp}

\DeclareMathOperator{\UAR}{\uparrow\negthickspace}

\def\preldiscuss{\def\d@scuss##1##2{\marginparsep=1em\marginparwidth=60pt
		\marginpar{\tt \tiny \raggedright\color{##1} ##2}}}
	
\def\d@scuss#1#2{\relax} 
\preldiscuss
\def\discuss#1{\d@scuss{black}{#1}}

\begin{document}

\maketitle

\begin{abstract}
We describe the structure of the monoid of natural-valued monotone functions on an arbitrary poset. For this monoid we provide a presentation, a characterization of prime elements, and a description of its convex hull. We also study the associated monoid ring, proving that it is normal, and thus Cohen-Macaulay. We determine its Cohen-Macaulay type, characterize the Gorenstein property, and provide a Gröbner basis of the defining ideal. Then we apply these results to the monoid of quasi-arithmetic multiplicities on a uniform matroid. Finally we state some conjectures on the number of irreducibles for the monoid of  multiplicities on an arbitrary matroid.
\end{abstract}


\section{Introduction}
Natural-valued monotone functions are ubiquitous in mathematics, and as we will see, they are tightly related to monotone Boolean functions. The study of monotone Boolean functions goes back at least to Dedekind \cite{dedekind}, and was continued by Church, Ward and others (see \cite{S-Y} and the references therein). This paper is devoted to the structure of the monoid of natural-valued monotone functions on an arbitrary finite poset $P$ (partially ordered set). 

If instead of monotone functions we consider order-reversing functions, then we come to the concept of $P$-partitions, that were studied by Stanley \cite[Chapter II]{stanley-osp}, \cite[Section~3.15]{stanley-ec1} and later by many other authors (see \cite{p-partitions} and the references therein). The study of both concepts, natural-valued monotone functions on a finite poset $P$ and $P$-partitions, is essentially equivalent.
In \cite{p-partitions}, the complete intersection property of the monoid ring of $P$-partitions was characterized in terms of forests with duplications, the graded ring of the monoid ring was described, and generating functions counting $P$-partitions were computed.

Our motivating example comes from matroid theory. Matroids axiomatize the linear algebra of lists of vectors. For instance, the uniform matroid $\rU(r,n)$ encapsulates the linear dependencies of a list of $n$ vectors in generic position in an $r$-dimensional space, that is, all the sublists of cardinality smaller than or equal to $r$ are linearly independent. Arithmetic matroids were introduced in \cite{m-t-p}, in relation with an invariant called the arithmetic Tutte polynomial \cite{MoT}, and since then proved to have a wide number of applications. Recent advances in understanding their structure have been achieved in \cite{DFM, Pag, PaPa}. A \emph{quasi-arithmetic matroid} is a matroid together with a suitable function called a \emph{multiplicity}; it is called \emph{arithmetic} if it satisfies an additional axiom. In \cite{DeMo}, Delucchi and the last author remarked that the set of quasi-arithmetic multiplicities on a given matroid $M$ is a monoid $\mathcal Q(M)$, and proved that arithmetic multiplicities form a submonoid $\mathcal A(M)$ of $\mathcal Q(M)$. 

In this paper we describe the structure of the monoid $\mathcal Q(M)$, with a special focus on the case of uniform matroids. Via an appropriate prime-wise slicing of the monoid, we can translate the problem of studying multiplicities on a given uniform matroid to the study of additive submonoids of a cartesian product of copies of the monoid $(\mathbb N, +)$. These submonoids are isomorphic to the set of monotone functions over a partially ordered set. 
\medskip

The paper is organized as follows. In Section~\ref{sec:monotone-functions} we study the structure of the monoid of natural-valued monotone functions on an arbitrary finite poset, describing its set of irreducibles (Proposition~\ref{prop:irred-pos}), a presentation of the associated monoid (Theorem~\ref{th:canonical-sg}, which is similar to \cite[Theorem 2.1]{p-partitions} for $P$-partitions), and the cone arising as its convex hull (Theorem~\ref{theo::supp}). We provide a  Gr\"obner basis for its defining ideal (Theorem~\ref{th:groebner-im}), characterize the Gorenstein property (Theorem~\ref{theo:Gor}), and describe the Cohen-Macaulay type of the monoid ring for certain posets (Proposition \ref{rem:type}) .
An irreducible monotone function over a partially ordered set rarely is a prime element of the monoid: in Theorems 
\ref{th:char-prime} and~\ref{th:char-prime-comb} we give a characterization of irreducible monotone functions that are prime.

In Section 3 we focus on the monoid of quasi-arithmetic multiplicities of a matroid. First we describe its structure in general, as a direct product of simpler monoids (Proposition \ref{infprod}); then we describe every factor in the case of uniform matroids (Proposition \ref{prop:gm-mf}). Then, by applying results of Section 2,  we determine when a slice of the set of multiplicity functions over an uniform matroid is Gorenstein (Theorem~\ref{prop:type}), and also compute the Cohen-Macaulay type in some extremal cases. Furthermore, in Theorem~\ref{prop:uni_prime} we characterize the irreducible and the prime elements of this monoid.

We finish our work by providing a couple of conjectures on the number of irreducibles, for which we have some experimental evidence.

\subsection*{Acknowledgements}
We are grateful to Marco D'Anna, Vic Reiner and Fengwei Zhou for helpful conversations. We also wish to thank an anonymous referee for many valuable remarks.

Experiments were performed with the help of  \texttt{Normaliz} \cite{normaliz}, \texttt{NormalizInterface} \cite{normalizinterface} and \texttt{NumericalSgps} \cite{numericalsgps} (the last two are GAP \cite{gap} packages). With these, we were able to run batteries of examples to foresee the results we proved later.

\section{Monotone functions on a partially ordered set}\label{sec:monotone-functions}

Given a partially ordered set $(P,\le)$ and a subset $X\subseteq P$, set 
$$
\UAR  X =\{y\in P\mid x\le y \hbox{ for some } x\in X\}.
$$
An \emph{upper set} in a partially ordered set $(P,\le)$ is a subset $U$ with the property that   $\UAR U= U$. 

For an ordered set $(P,\le)$, a function $f:P\to \mathbb N$ is \emph{monotone} if $f(a)\le f(b)$ whenever $a\le b$ ($\mathbb{N}$ denotes the set of nonnegative integers). By $\rM(P)$ we denote the set of monotone functions over $P$. 

Clearly, if $U$ is an upper subset of $P$, then $\chi_U\in \rM(P)$. Here $\chi_U$ is the \emph{indicator function} of $U$: $\chi_U(x)=1$ for $x\in U$, and $\chi_U(x)=0$ otherwise.

We can define on $\rM(P)$ addition as $(f+g)(p)=f(p)+g(p)$ for all $p\in P$. Under this binary operation, $\rM(P)$ is a commutative monoid.

A nonempty upper set is \emph{irreducible} if it is not the union of two disjoint upper sets.

\begin{lemma}\label{lem:exp-union-irr}
Let $(P,\le)$ be a finite partially ordered set, and let $U$ be an upper set of $P$. Then $U$ can be expressed (uniquely) as the disjoint union of irreducible upper sets of $P$.
\end{lemma}

\begin{proof}
We proceed by induction on $|U|$. If $U$ is not irreducible, then it is the union of two disjoint upper sets, which are disjoint union of irreducible upper sets by inductive hypothesis. 

The uniqueness of the decomposition is obvious: an irreducible upper set $V$ is only contained in the disjoint union of upper sets $W_i$ if $V\subset W_i$ for some $i$.
\end{proof}

For a map $f:P\to \mathbb{N}$, its \emph{support} is defined as 
$$
\supp(f)=\{p \in P \mid f(p)\neq 0\}. 
$$ 
Clearly, if $f$ is monotone, then $\supp(f)$ is an upper set of $P$. This fact, together with Lemma~\ref{lem:exp-union-irr}, is the key to see what  the irreducibles of $\rM(P)$ are, that is, monotone functions over $P$ that cannot be expressed as the sum of two other nonzero monotone functions over $P$.

\begin{proposition}\label{prop:irred-pos}
Let $(P,\le)$ be a finite partially ordered set. The monoid $(\rM(P),+)$ is minimally generated by 
$$
\bigl\{\chi_I\mid I\hbox{  irreducible upper set of }P\bigr\}.
$$ 
Its rank is $|P|$.
\end{proposition}

\begin{proof}
Let $f$ be a monotone function over $P$. We already know that $\supp(f)$ is an upper set. By Lemma~\ref{lem:exp-union-irr}, there exists a family $\cI$ of upper sets such that $\supp(f)=\dot\bigcup_{I\in\cI} I$ (disjoint union). Then $\chi_{\supp(f)} = \sum_{I\in \cI} \chi_I$. Clearly, $f'=f-\chi_{\supp(f)}$ is a monotone function over $P$. We can repeat the process with $f'$ until we reach the zero function (this process will stop since $P$ has finitely many elements and the values of $f$ are nonnegative integers). 

This shows that $\rM(P)$ is generated by the set of elements $\chi_I$ with $I$ an irreducible upper set. It is also clear that every $\chi_I$, with $I$ an irreducible upper set, is an irreducible. Being contained in $\mathbb{N}^{|P|}$, the monoid $\rM(P)$ is minimally generated by its irreducibles. This can be seen as follows. We set 
$$
\deg(f)=\sum_{x\in P} f(x).
$$ 
If $f\neq 0$ is not irreducible, then $f=g+h$ with $g,h\neq 0$. Since $\deg(g), \deg (h)<\deg(f)$, it follows by induction that the irreducibles generate $\rM(P) $, and it is clear that the irreducibles are contained in every system of generators (see also \cite[Chapter 3, Exercise 6]{f-g}).

Since $\rM(P)\subset \ZZ^{|P|}$, its rank is at most $|P|$. On the other hand, there is a strictly ascending chain of upper sets of length $|P|$:  we take a linear refinement of the partial order that lists $P=\{a_1,\dots,a_n\}$ in ascending order, and consider the upper sets $\{a_i,\dots,a_n\}$, $i\in\{1,\dots n\}$. Their characteristic functions are linearly independent.
\end{proof}

We can sharpen this last proposition to obtain a canonical expression of a monotone function in terms of the $\chi_I$ with $I$ an irreducible upper set. To this end we need to introduce the concept of near-chain.

\subsection{Near-chains}

A \emph{near-chain} of irreducible upper sets in a partially  ordered set $(P,\le)$ is a set $\cI =\{I_1,\dots, I_m\}$ of irreducible upper sets such that one of the following relations holds for all $I_i,I_j\in\cI$: $I_i\subset I_j$, $I_j\subset I_i$ or $I_i\cap I_j=\emptyset$. In \cite{p-partitions}, the corresponding concept of near-chain for order ideals is a multiset of nonempty connected order ideals that pairwise intersect trivially.
The following result has an analogue for $P$-partitions; see Section 1.2 in \cite{p-partitions}.

\begin{proposition}\label{prop:exp-nc}
Let $(P,\le)$ be a finite partially ordered set. Let $f\in \rM(P)$. Then there exist a unique near-chain $\cI$ and uniquely determined positive integers  $a_I$, for $I\in\cI$, such that
$$
f=\sum_{I \in\cI} a_I \chi_I.
$$
\end{proposition}

\begin{proof}
Let $S=\supp(f)$. Then $S$ has a unique representation as the union of disjoint irreducible upper sets, say $S=\dot\bigcup_{J\in \cK} J$ (Lemma~\ref{lem:exp-union-irr}). Thus $\chi_S=\sum_{J\in\cK} \chi_J$. 

For $f=0$ we take $\cI=\emptyset$. For $f\neq 0$ we can assume that the assertion on existence holds for $f-\chi_S$ since $\deg(f-\chi_S)<\deg(f)$ where $\deg$ is defined as above. Any irreducible upper set appearing in the representation $f-\chi_S=\sum_{I \in\cI'} a_I' \chi_I$ must be contained in one of the disjoint irreducible upper sets $J\in\cK$, and therefore $\cI = \cK \cup \cI'$ is a near-chain. Set $a_I=1$ if  $I\in \cK\setminus \cI'$, $a_I = a_I'$ for $I\in \cI'\setminus \cK$ and $a_I=a_I'+1$ for $I\in \cK\cap \cI'$.

Uniqueness is proved similarly. If we have a representation $f=\sum_{I \in\cI} a_I \chi_I$ with a near-chain $\cI$, then $S$ is the union of the maximal elements in $\cI$. Therefore they form the set $\cK$ from above, and we are again done by induction.
\end{proof}

\subsection{Relations and defining ideal}
We already know the generators of $\rM(P)$ for $P$ a finite partially ordered set. Let us now describe this monoid in terms of generators and relations. 

Let $\mathcal{F}$ be the free monoid on the irreducible upper sets of $P$. Then the morphism $\varphi$ from $\mathcal{F}$ to $\rM(P)$ induced by $I\mapsto \chi_I$ is an epimorphism, and $\rM(P)$ is isomorphic to $\mathcal{F}/\ker\varphi$, where $\ker\varphi$ is the kernel congruence of $\varphi$, that is, the set of pairs $(x,y)$ such that $\varphi(x)=\varphi(y)$. A system of generators of $\ker\varphi$ is known as a \emph{presentation} of $\rM(P)$. A system of generators that solves the word problem for a given admissible total order on $\mathcal{F}$ is known as a canonical basis of $\ker\varphi$ (see \cite{f-g}).

Let $\cI$ be a family of irreducible upper sets, and let $a_I$ be a positive integer for all $I\in\cI$. We define the \emph{degree} of the expression $\sum_{I\in\cI} a_I \chi_I$ as 
$$
\deg\left(\sum\nolimits_{I\in \cI}a_I \chi_I\right)= \sum_{I\in\cI}a_I \lvert I\rvert.
$$
In particular we have $\deg(\chi_I) = \vert I\vert$. Note that this definition of degree is consistent with the one in the proof of Proposition \ref{prop:irred-pos}: when we specialize the formal expression to a function on $P$, the degree stays the same.

We fix a total order $\prec$ of all $\chi_I$ with $I$ irreducible such that $\deg \chi_I\le \deg \chi_J$ implies $\chi_i \preceq \chi_J$. We then extend it to an order on the set of all formal expressions $\sum\nolimits_{I\in \cI}a_I \chi_I$. We write $\sum\nolimits_{I\in \cI}a_I \chi_I\preceq \sum\nolimits_{J\in \cJ}a_J \chi_J$ with $\mathcal I$ and $\mathcal J$ sets of irreducible upper sets and $a_I$, $a_J$ positive integers, whenever 
\begin{enumerate}
\item $\deg(\sum\nolimits_{I\in \cI}a_I \chi_I)<\deg(\sum\nolimits_{J\in \cI}a_J \chi_J)$, or
\item  $\deg(\sum\nolimits_{I\in \cI}a_I \chi_I)=\deg(\sum\nolimits_{J\in \cJ}a_J \chi_J)$ and $\sum\nolimits_{I\in \cI}a_I \chi_I$ is smaller than or equal to $\sum\nolimits_{J\in \cJ}a_J \chi_J$ with respect to the reverse lexicographical order induced by the order of the $\chi_I$.
\end{enumerate}
In total we have defined a degree reverse lexicographical order. Note that degree is evaluated first. For example, if $|I| < |J| < 2|I|$, then $\chi_I\prec \chi_ J \prec 2\chi_I$.

Let $I$ and $J$ be two irreducible upper sets such that one is not contained in the other and they have nonempty intersection, that is, $\{I,J\}$ is not a near-chain of $(P,\le)$. Then $\chi_I+\chi_J=\chi_{I\cap J}+\chi_{I\cup J}$. We can then express $I\cap J$ as a disjoint union of irreducible upper sets, say $I\cap J = \dot\bigcup_{U\in \mathcal{U}} U$, and the same for $I\cup J=\dot\bigcup_{V\in \mathcal{V}} V$. Then 
\begin{equation}\label{eq:no-nc}
\chi_I+\chi_J=\sum_{U\in \mathcal{U}}\chi_U+ \sum_{V\in \mathcal{V}}\chi_V.
\end{equation}

Notice that $\chi_I+\chi_J$ is larger than  $\sum_{U\in \mathcal{U}}\chi_U+ \sum_{V\in \mathcal{V}}\chi_V$ with respect to $\preceq$, since for all $U\in \mathcal{U}$, $\lvert U \rvert <\min\{\lvert I\rvert,\lvert J\rvert\}$.

If $f\in \rM(P)$, with $(P,\le)$ a partially ordered set, then by Proposition~\ref{prop:irred-pos}, $f$ admits an expression of the form $f=\sum_{I\in\cI} a_I \chi_I$ for some set of irreducible upper sets $\cI$ and some positive integers $a_I$. If $\cI$ is not a near-chain, then there is some $I,J\in\cI$ such that $I\cap J$ is not empty and neither $I\subseteq J$ nor $J\subseteq I$. We can replace $\chi_I+\chi_J$ in the expression $\sum_{I\in\cI} a_I \chi_I$ with $\sum_{U\in \mathcal{U}}\chi_U+ \sum_{V\in \mathcal{V}}\chi_V$. With the new expression we repeat the process. Every time we apply a substitution we are replacing a sum of two irreducibles by another sum with smaller order. Thus, after a finite number of steps, this process will stop, obtaining the canonical expression of $f$ given in Proposition~\ref{prop:exp-nc}, which is the normal form with respect to $\preceq$. 

This in particular shows that the set of pairs  $(I+J, \sum_{U\in \mathcal{U}}U+ \sum_{V\in\mathcal{V}}V)$ corresponding to \eqref{eq:no-nc} solves the word problem in $\mathcal{F}$: in order to decide if two expressions $\sum_{I\in\cI} a_I\chi_I$ and $\sum_{J\in\cJ} a_J\chi_J$ represent the same element in $\rM(P)$ (they map to the same element via $\varphi$), we only have to compute their canonical expressions and see if they coincide.  Thus we have shown the following theorem.

\begin{theorem}\label{th:canonical-sg}
Let $(P,\le)$ be a finite partially ordered set. Let $\mathcal{F}$ be the free monoid on the irreducible upper sets of $P$. Then the morphism $\varphi$ from $\mathcal{F}$ to $\rM(P)$ induced by $I\mapsto \chi_I$ is an epimorphism. Moreover, the set of pairs 
$$
\left(I+J, \sum_{U\in \mathcal{U}}U+ \sum_{V\in\mathcal{V}}V\right),
$$ 
for every $I$ and $J$ with $\{I,J\}$ not being a near-chain, where $I\cap J$ and $I\cup J$ decompose as a disjoint union of irreducibles as $I\cap J=\dot\bigcup_{U\in\mathcal U} U$ and $I\cup J=\dot\bigcup_{V\in\mathcal V} V$, is a canonical basis for the kernel congruence of $\varphi$ for the order $\preceq$.
\end{theorem} 

Let $K$ be a field, and let $t$ be a symbol. The \emph{monoid ring} of $\rM(P)$ is defined as $K[\rM(P)]=\bigoplus_{f\in \rM(P)} K t^f$, where addition is performed component-wise and multiplication is determined by the rule $t^ft^g=t^{f+g}$ and the distributive law. For every irreducible upper set $I$, take a variable $x_I$, and let $R$ be the polynomial ring on these variables with coefficients in $K$. Then we can define the ring homomorphism determined by the images of $x_I$ for all $I$
$$
\psi: R\to K[\rM(P)],\ x_I\mapsto t^{\chi_I}.
$$
The kernel of $\psi$ is known as the \emph{ideal associated} to $K[\rM(P)]$, denoted $\rI_{\rM(P)}$. By Herzog's correspondence, \cite{herzog},
$$
\rI_{\rM(P)}=\left\{x_{I_1}^{a_1}\cdots x_{I_n}^{a_n}- x_{J_1}^{b_1}\cdots x_{J_m}^{b_m} \mid (a_1\chi_{I_1}+ \dots + a_n\chi_{I_n}, b_1\chi_{J_1}+ \dots + b_m\chi_{J_m} ) \in \ker\varphi \right\}.
$$

We can define the degree of $x_I$ as $\lvert I\rvert$, and if two variables have the same degree, we can arrange them as we arranged $\chi_I$ above. Then $\preceq$ translates to a monomial ordering on $R$, and the paragraphs preceding Theorem~\ref{th:canonical-sg} prove the following result.

\begin{theorem}\label{th:groebner-im}
Let $(P,\le)$ be a finite partially ordered set. Let $B$ be the set of binomials 
$$
x_Ix_J-\prod_{U\in \mathcal U} x_U\prod_{V\in \mathcal  V}x_V,
$$ 
such that $\{I,J\}$ is not a near-chain of $(P,\le)$, and $\mathcal U$ and $\mathcal V$ are partitions of irreducible upper sets of $I\cap J$ and $I\cup J$, respectively. Then $B$ is a Gr\"obner basis of the ideal $\rI_{\rM(P)}$ with respect to the order $\preceq$.
\end{theorem}

Let $\mathcal M$ be the monomial ideal generated by the $x_Ix_J$ with $\{I,J\}$ not a near-chain. The complementary set of monomials in $R$ are exactly those whose support is a near-chain. Moreover these monomials are linearly independent (as a consequence of Proposition~\ref{prop:exp-nc}). This gives an alternative proof of Theorem~\ref{th:groebner-im}. In \cite[Theorem 1.2]{p-partitions} a system of generators of the ideal associated to the monoid ring of $P$-partitions is given.

The upper sets in $P$ form a distributive lattice $\mathcal{L}$ with respect to intersection and union. With such a lattice one can associate its Hibi ring $K[\mathcal{L}]$ \cite{Hibi} that as a $K$-algebra is defined by the generators $x_I$, $I\in\mathcal{L}$, and the relations
$$
x_Ix_J=x_{I\cap J}x_{I\cup J}.
$$
The Hibi ring is standard graded with $\deg x_I=1$ for all $I$ since all its defining relations are homogeneous of degree $2$.

\begin{theorem}
Let $(P,\le)$ be a finite partially ordered set. Then the following hold:
\begin{enumerate}[1.]	
\item $K[\rM(P)]$ is an integral domain of Krull dimension equal to $|P|$. 
\item $K[\rM(P)]$ is the dehomogenization of $K[\mathcal{L}]$ with respect to the degree $1$ element $x_\emptyset$, that is, $K[\rM(P)]\cong K[\mathcal{L}]/(x_\emptyset-1)$.
\end{enumerate}
\end{theorem}

\begin{proof}
The first statement follows from the general fact that the $K$-algebra $K[M]$ for an arbitrary affine monodid $M$ is an integral domain of Krull dimension equal to $\operatorname{rank} M$. That $\operatorname{rank}\rM(P) =|P|$ has been stated in Proposition \ref{prop:irred-pos}.

For the second statement we observe that in $K[\rM(P)]$ we have $t^{\chi_I}t^{\chi_J}=t^{\chi_{I\cap J}}t^{\chi_{I\cup J}}$, and $t^{\chi_\emptyset}=1$. The substitution $x_I\mapsto t^{\chi_I}$ therefore induces  a surjective algebra homomorphism $K[\mathcal L]\to K[\rM(P)]$ whose kernel contains $x_\emptyset-1$. Some immediate observations: (i) $x_\emptyset-1\neq 0$ since it is the difference of elements of different degrees; (2) it is even a nonzerodivisor since every zerodivisor in the graded ring $K[\mathcal{L}]$ is annihilated by a nonzero homogeneous element, and such an element would have to annihilate $1$ (see \cite[1.5.6]{BH}).

The main point:  $x_\emptyset-1$ generates a prime ideal of height $1$ in the integral domain $K[\mathcal L]$ by the general properties of dehomogenization (for example, see \cite[p. 38]{BH}), and since the Krull dimensions of $K[\mathcal L]$ and $K[\rM(P)]$ differ by $1$, one has the isomorphism $K[\rM(P)]\cong K[\mathcal{L}]/(x_\emptyset-1)$. In fact, in any Noetherian ring of finite Krull dimension and for every prime ideal $Q$ of $R$ one has $\dim R \ge \dim R/Q + \operatorname{height} Q$. This excludes the possibility that the kernel of the homomorphism $K[\mathcal L]\to K[\rM(P)]$ properly contains $(x_\emptyset-1)$.  
\end{proof}

\subsection{Convex hull}

Let $(P,\le)$ be a finite partially ordered set. We say that $y\in P$ is a \emph{cover} of $x\in P$ if $x<y$, but there is no $z\in P$ such that $x<z<y$.	Let $\rC(P)$ be the set of monotone functions $f:P \to \mathbb{Q}$.

Observe that $\rM(P)\subset \rC(P)$, and that for every $q\in \rC(P)$, there exists a positive integer $n$ such that $nq\in \rM(P)$. Thus $\rC(P)$ is the convex closure of $\rM(P)$ in $\mathbb{Q}^{|P|}$. 

We can identify $\rM(P)$ with the set of vectors $(y_p)_{p\in P}\in \NN^{\lvert P\rvert}$ such that $y_p\le y_q$ whenever $p\le q$. Thus $\rM(P)$ can be seen as a normal submonoid of $\NN^{\lvert P\rvert}$, and $\rC(P)$ corresponds to the cone spanned by it in $\QQ^{\lvert P\rvert}$. We now list the extremal rays and the support hyperplanes of $\rC(P)$, a fact that implicitly appears in \cite[Section 7]{l-e}, but we include here with our notation for sake of completeness. 

\begin{theorem} \label{theo::supp}
The set of extremal rays of $\rC(P)$ is 
$$
\bigl\{\chi_I\mid I\hbox{  irreducible upper set of }P\bigr\}.
$$
Moreover, the cone $\rC(P)$ is cut out from the space of all functions $f:P\to\QQ$ by the inequalities
\begin{align}
f(x)&\ge 0,\qquad &&x\in P, \text{ $x$ minimal in $P$}	,\label{equ:pos}\\
f(x)&\le f(y),\qquad &&x,y\in P, \text{ $y $ is a cover of $x$},\label{equ:cov}
\end{align}
and this description is minimal.
\end{theorem}

\begin{proof}
Since the set $\{\chi_I\mid I\hbox{  irreducible upper set of }P\}$ generates $\rM(P)$, its $\QQ^+$-linear span is $\rC(P)$. We must show that an equation
$$
\chi_I=\sum_{i=1}^{r} a_i \chi_{I_i}, \qquad a_i>0 \text{ for all } i,
$$
with $I,I_1,\dots,I_r$ irreducible is only possible with $I_j=I$ for all $j$. (Clearly $I_j\subset I$ for all $j$.)

Assume the contrary. Then there is $k\in I$ such that $k\notin I_j$ for some $j$. We can assume $k\in I_1,\dots,I_t$, $k\notin I_{t+1},\dots, I_r$. By looking at the value of $\chi_I$ in $k$, we see that
$$
\sum_{i=1}^{t}a_i=1. 
$$
Let $J=\bigcap_{i=1}^t I_i$. Then at each element $j$ in $J$ the sum  $\sum_{i=1}^t a_i\chi_{I_i}(j)=1$. But this implies that $J\cap (I_{t+1}\cup\dots\cup I_r)=\emptyset$. In fact, if $h\in J\cap (I_{t+1}\cup\dots\cup I_r)$, then $\chi_I(h)=\sum_{i=1}^r a_i\chi_{I_i}(h)>1$, a contradiction.

We see now that $I=J\cup (I_{t+1}\cup\dots\cup I_r)$. Namely, if $h\in I$, $h\notin J$, then not all $a_j\chi_{I_j}$, $j\in \{1,\dots,t\}$, contribute to the value of $\chi_I$ in $h$. So at least one of the $\chi_{I_u}$ with $t+1\le u\le r$ must be equal to $1$ in $h$, and so $h\in I_u$.

This contradicts the irreducibility of $I$, since $J$ and $I_{t+1}\cup\dots\cup I_r$ are non\-empty upper sets, so the first statement is proved.	

As for the support hyperplanes, it is evident that exactly the monotone functions with nonnegative values satisfy the set of inequalities \eqref{equ:pos} and \eqref{equ:cov}, and that one cannot omit any of the inequalities in \eqref{equ:pos}. Only minimality of the inequalities in \eqref{equ:cov} could be an issue. To this end, let us fix $x$ and $y$ such that $y$ is a cover of $x$, and define the function $f:P\to \QQ$ by $f(z)=1$ for all $z\in \UAR \{x\}$, $z\neq y$, and $f(z)=0$ elsewhere. Then $f$ is not monotone, but satisfies all inequalities except $f(x)\le f(y)$.	
\end{proof}

For a near-chain $\Gamma=\{I_1,\dots,I_m\}$ of irreducible upper sets  we take
$$
\sigma_\Gamma=\rL_{\QQ_+}(\{\chi_{I_1},\dots,\chi_{I_m}  \}),
$$
the set of all $\QQ$-linear combinations with nonnegative coefficients. We have already seen that $\chi_{I_1},\dots,\chi_{I_m}$ are linearly independent (as a consequence of Proposition~\ref{prop:exp-nc}). Therefore $\sigma_\Gamma$ spans a simplicial cone.

\begin{corollary}
Let $(P,\le)$ be a finite partially ordered set. Then the collection $(\sigma_\Gamma)$, $\Gamma$ a near-chain of irreducible upper sets, is a unimodular triangulation of $\rC(P)$. 
\end{corollary}

This follows from Theorem~\ref{th:groebner-im} by the Sturmfels correspondence (\cite[Corollary 7.20]{BrGu}). It is also an immediate consequence of Proposition~\ref{prop:exp-nc}.

In the terminology of toric algebra, Proposition~\ref{prop:irred-pos} says that the functions $\chi_I$, for irreducible upper sets $I$, form the Hilbert basis of $\rC(P)$.

\subsection{Cohen-Macaulay type}

One says that $P$ is \emph{graded} if there exists a \emph{level function} $\gamma:P\to \ZZ$ such that (i) $\gamma(x)=1$ if $x$ is a minimal element of $P$ and (ii) $\gamma(y)=\gamma(x)+1$ whenever $y$ is a cover of $x$. (It is more customary to assume that $\gamma(x)=0$ for minimal $x$, but the two cases are equivalent since we can add a constant without changing condition (ii).)  Evidently $\gamma$ is uniquely determined. An equivalent condition is that all maximal chains connecting an element $y$ and any minimal element $x\le y$ have the same length.

\begin{theorem}\label{theo:Gor}
Let $(P,\le)$ be a finite partially ordered set. Then the following hold:
\begin{enumerate}[1.]
\item the ring $K[\rM(P)]$ is a normal Cohen-Macaulay domain;
\item \label{item:gorenstein} it is Gorenstein if and only if $P$ is graded;
\item under the equivalent conditions of~\ref{item:gorenstein}, the generator of the canonical module is the level function.
\end{enumerate}
\end{theorem}

\begin{proof}
Already by its definition, the monoid $\rM(P)$ is the set of lattice points in a rational cone. Therefore it is normal, and the algebra $K[\rM(P)]$ is Cohen-Macaulay by Hochster's Theorem \cite{hochster}.

The Gorenstein property of $K[\rM(P)]$ is equivalent to the existence of a function $\gamma:P\to \ZZ$ such that all linear forms in the inequalities \eqref{equ:pos} and \eqref{equ:cov} that represent the support hyperplanes of $\rC(P)$ have value $1$ on $\gamma$ (see \cite[Theorem 6.33]{BrGu}), that is, $\gamma(x)=1$ for every $x$ minimal in $P$, and $\gamma(y)-\gamma(x)=1$ if $y$ is a cover of $x$. But this is exactly the condition that $\gamma$ is a level function on $P$.

The last statement follows from the theorem of Danilov and Stanley that we briefly discuss  for the general case below.
\end{proof}

Since the Gorenstein property depends only on $\rM(P)$ (and not on $K$) we say that $\rM(P)$ is \emph{Gorenstein} if $K[\rM(P)]$ is a Gorenstein ring. The same convention can be used for the type that we discuss now. The \emph{Cohen-Macaulay type} of a normal monoid is the cardinality of the set of minimal generators of the interior of the cone defining it. In fact, the type is the minimal number of generators of the canonical module \cite[Prop. 3.3.11]{BH}, and by a theorem of Danilov and Stanley the canonical module is in our case the ideal generated by the monomials that correspond to lattice points in the interior of the cone \cite[Theorem 6.31]{BrGu}. The interior of the cone is an ideal of the monoid, and it is determined by turning the defining inequalities of the monoid into strict inequalities. So the type is the cardinality of the set $\operatorname{Minimals}_{\le_{\rM(P)}}(\operatorname{int}(\rC(P))\cap\rM(P))$ for which the order $\le_{\rM(P)}$ is defined as follows: $ f \le_{\rM(P)} g$ if $g-f\in \rM(P)$.

\begin{example}\label{ex:type}
Consider the poset $P$ with Hasse diagram
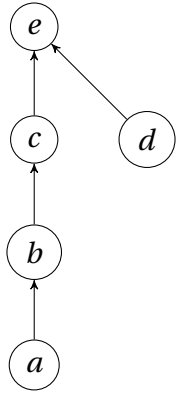
\begin{center}
\begin{tikzpicture}[->,>=stealth',auto,node distance=1.5cm, main node/.style={circle,draw}]

\node[main node] (1) {$a$};
\node[main node] (2) [above of=1] {$b$};
\node[main node] (3) [above of=2] {$c$};
\node[main node] (4) [right of=3] {$d$};
\node[main node] (5) [above of=3] {$e$};

\path[every node/.style={font=\sffamily\small}]
(1) edge node  {} (2)
(2) edge node  {} (3)
(3) edge node  {} (5)
;

\path[every node/.style={font=\sffamily\small}]
(4) edge node  {} (5);

\end{tikzpicture}
\end{center}
Then $\rM(P)$ is generated by 
$$
\{\chi_{e}, \chi_{de},\chi_{ce},\chi_{cde},\chi_{bce},\chi_{bcde},\chi_{abce},\chi_{abcde}\}.
$$
While the interior is 
$$
\{g_1,g_2,g_3\}+\rM(P)
$$
where $g_1\equiv ( 1, 2, 3, 1, 4 )$, $g_2\equiv ( 1, 2, 3, 2, 4)$, and $g_3\equiv ( 1, 2, 3, 3, 4)$. So the Cohen-Ma\-caulay type of $\rM(P)$ is $3$.
\end{example}

For the application in Theorem \ref{prop:type} we describe the minimal set of generators of $\operatorname{int}(\rC(P))\cap\rM(P)$  for a special type of poset that generalizes Example \ref{ex:type}. It uses the \emph{length} of a chain in a poset, which for us is the number of elements in  the chain.

\begin{proposition}\label{rem:type}
Let $P$ be a poset satisfying the following conditions:
\begin{enumerate}[1.]
\item for all $x\in P$, all inextensible chains ascending from  $x$ have the same length;
\item for all maximal elements $x\in P$ the maximal length of a chain descending from $x$ is independent  of $x$.
\end{enumerate}
Then a strictly monotone function $\gamma: P\to \NN\setminus\{0\}$ is minimal in $\operatorname{int}(\rC(P))\cap\rM(P)$ if and only if 
$\gamma(x_i)=i $ for all chains $x_1<\dots<x_m$ of maximal length.		
\end{proposition}

\def\coht{\operatorname{coht}}

\begin{proof}
First we show sufficiency. By 2 all chains of maximal length have the same length $m$ that appears in the conclusion.	For a chain $x_1<\dots<x_m$ we must have $\gamma(x_i) \ge i$ since $\gamma$ is strictly monotone and ${\gamma(x_1)\ge }1$. Therefore the value $\gamma(x_i)=i$ is the smallest possible. By 2 this implies $\gamma(x)=m$ for all maximal elements $x$.  It follows immediately that $\gamma - \delta\notin \operatorname{int}(\rC(P))\cap\rM(P)$ for nonzero $\delta \in \rM(P)$ since  $\delta(x)\ge 1$ for at least one maximal element $x$.

In order to prove necessity, we define the coheight of $y\in P$ by 
$$
\coht(y) =\max\bigl\{k\mid \text{there exists a chain } y <x_1<\dots< x_k \text{ in } P \bigr\}.
$$
Every maximal element has coheight $0$. Suppose that $x>y$ is a cover of $y$. Then $\coht(y) =\coht (x)+1$, as follows from condition 1: every inextensible chain joining $x$ to a maximal element $z$ extends to such a chain from $y$ to $z$. The length of the chain increases by $1$.

We must show that $\gamma(z) = m$ for all maximal elements $z$ of $P$ if $\gamma - \delta$ is not in the interior of $C(P)$ for any nonzero monotone $\delta$. Assume that $\gamma(z)> m$ for at least one maximal element $z$. Then the set
$$
U = \bigl\{x\in P\mid  \gamma(x) > m -\coht(x)   \bigr\}
$$
contains $z$ and is therefore nonempty. We claim that $U$ is an upper set. Suppose that $x\in U$ and $y>x$. We must show that $y\in U$ as well. It is enough to consider a cover $y$ of $x$. Then
$$
\gamma(y) > \gamma(x) \ge  m -\coht(x) +1 = m - \coht(y).
$$
Next we claim that $\gamma(x) \le \gamma(y) -2$ if $x\notin U$, $x< y$ and $y \in U$.  In fact,
$$
\gamma(x) \le m - \coht(x) \le m - \coht(y) -1  \le \gamma(y) -2.
$$
It follows that $\gamma - \chi_U$ is still strictly monotone. It takes only positive values since $m-\coht(x)\ge 1$ for all $x\in P$, and therefore $x\notin U$ if $\gamma(x) = 1$. Thus $\gamma - \chi_U$ belongs to the interior of $C(P)$.	
\end{proof}

In \cite{p-partitions} a characterization of the complete intersection property of the monoid ring of $P$-partitions is given. 

\subsection{Prime elements}

Let $(M,+)$ be a commutative monoid. We say that $a\le_M b$ if there exists $c$ such that $a+c=b$ (recall that we already defined this relation for $M=\rM(P)$). If the monoid $M$ is cancellative ($a+b=a+c$ implies $b=c$) and reduced ($a+b=0$ implies $a=b=0$; no nontrivial units), then $\le_M$ is an order relation. We will say that $a$ \emph{divides} $b$ if $a\le_M b$.

An element $a$ in $M$ is a \emph{prime} if whenever $a\le_M (b+c)$, with $b,c\in M$, then either $a\le_M b$ or $a\le_M c$. Observe that prime elements are irreducible, while in general irreducible elements do not need to be prime.

If $(P,\le)$ is a finite partially ordered set, we can define on $\rM(P)$ the following order relation: $f\le g$ if $f(a)\le g(a)$ for all $a\in P$. Notice that 
$$
f\le_{\rM(P)} g\quad \hbox{implies}\quad f\le g.
$$

\begin{example}\label{ex:two-max-two-min}
Let $P$ be the partially ordered set with Hasse diagram
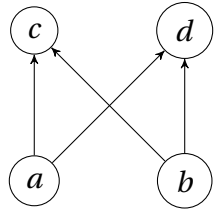
\begin{center}
\begin{tikzpicture}[->,>=stealth',auto,node distance=2cm, main node/.style={circle,draw}]

\node[main node] (1) {$a$};
\node[main node] (2) [right of=1] {$b$};
\node[main node] (3) [above of=1] {$c$};
\node[main node] (4) [above of=2] {$d$};

\path[every node/.style={font=\sffamily\small}]
(2) edge node  {} (3);

\path[every node/.style={font=\sffamily\small}]
(2) edge node  {} (4);

\path[every node/.style={font=\sffamily\small}]
(1) edge node {} (3);
\path[every node/.style={font=\sffamily\small}]
(1) edge node {} (4);
\end{tikzpicture}
\end{center}

We identify each monotone function over $P$ with a tuple $(a,b,c,d)$, where each coordinate is the corresponding value of the function in the node. The set of monotone functions can then be viewed as the set of nonnegative integer solutions to the inequalities 
$$
a\le c,\ a\le d,\ b\le c,\ b\le d.
$$
(See \cite[Section 4.5]{stanley-ec1} for an approach for $P$-partitions with the use of slack variables.)
The set of irreducibles is 
$$\{ (0,0,0,1),(0,0,1,0),(0,1,1,1),(1,0,1,1),(1,1,1,1)
\},
$$
which correspond respectively to the characteristic functions of the irreducible upper sets
$$
\{d\}, \{c\}, \{b,c,d\}, \{a,c,d\}, \{a,b,c,d\}.
$$
Notice that $\{\{a,c,d\},\{b,c,d\}\}$ is not a near-chain, and so we get an expression like in Equation \eqref{eq:no-nc}
$$
(1,0,1,1)+(0,1,1,1)=(1,1,1,1)+(0,0,1,0)+(0,0,0,1).
$$
As all the irreducibles appear in one of the two sides of the above equality, we deduce that $\rM(P)$ has no primes.
\end{example}

This example motivates the following characterization of prime elements: a prime element cannot appear in any of the sides of Equation \eqref{eq:no-nc}.

\begin{theorem}\label{th:char-prime}
Let $(P,\le)$ be a finite partially ordered set, and let $I$ be an irreducible upper set on $P$. 
Then $\chi_I$ is a prime element if and only if the following conditions hold:
\begin{enumerate}[1.]
\item $\chi_I\not\le_{\rM(P)}\chi_{G\cup H}$ and  $\chi_I \not\le_{\rM(P)} \chi_{G\cap H}$ for any irreducible upper sets $G,H$ with $G\neq I\neq H$;
\item if $J$ is an irreducible upper set other than $I$, then $\{I,J\}$ is a near-chain.
\end{enumerate}
\end{theorem}

\begin{proof}
Let us first show the necessity of the conditions.
\begin{enumerate}[1.]
\item  We have
$$
\chi_G+\chi_H=\chi_{G\cap H}+\chi_{G\cup H}.
$$
Therefore, if $\chi_I$ divides one of the two ``factors'' on the right hand side, it must divide $\chi_G$ or $\chi_H$, which is impossible since they have no proper divisors.

\item Assume to the contrary that there exists an irreducible upper set $J\neq I$ such that $\{I,J\}$ is not a near-chain. Then we can find an expression like \eqref{eq:no-nc}, $\chi_I+\chi_J=\sum_{U\in \mathcal{U}}\chi_U+ \sum_{V\in \mathcal{V}}\chi_V$. This means that $\chi_I\le_{\rM(P)} \sum_{U\in \mathcal{U}}\chi_U+ \sum_{V\in \mathcal{V}}\chi_V$, but then $\chi_I\le_{\rM(P)} \chi_H$ for some $H\in \mathcal{U}\cup\mathcal{V}$, which is impossible. 
\end{enumerate}

Now we turn to sufficiency. Assume that $\chi_I\le_{\rM(P)} f+g$, for some $f,g\in \rM(P)$. Then there exists $h\in\rM(P)$ such that $\chi_I+h=f+g$. Write $h=\sum_{H\in \mathcal{H}} a_H\chi_H$ with $\mathcal{H}$ a near-chain (Proposition~\ref{prop:exp-nc}). By condition 2, we have that $\mathcal{H}\cup\{I\}$ is a near-chain. Thus the normal form of $\chi_I+h$ is $(a_I+1)\chi_I+\sum_{H\in \mathcal{H}\setminus\{I\}} a_H\chi_H$ (where we set $a_I=0$ in case $I\not\in \mathcal{H}$). Assume that  $\chi_I$ appears neither in any expression of $f$ nor in any expression of $g$ (if this is not the case, then we are done). In light of Theorem~\ref{th:canonical-sg}, by applying replacements like the ones given in Equation \eqref{eq:no-nc}, we should be able to go from the expression of $f+g$ to the normal form $(a_I+1)\chi_I+\sum_{H\in \mathcal{H}\setminus\{I\}} a_H\chi_H$. But this implies that at some point  $\chi_I$ will divide either $\chi_{G\cup H}$ or $\chi_{G\cap H}$ for some $H, G$ irreducible upper sets different from $I$, contradicting condition 1.	
\end{proof}

\begin{example}\label{ex:with-ab-a}
Let $P$ be the partially ordered set with Hasse diagram
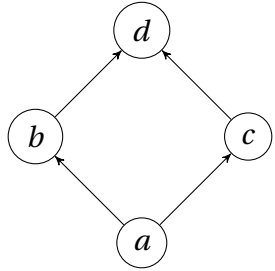
\begin{center}
\begin{tikzpicture}[->,>=stealth',auto,node distance=2cm, main node/.style={circle,draw}]

\node[main node] (1) {$a$};
\node[main node] (2) [above left of=1] {$b$};
\node[main node] (3) [above right of=1] {$c$};
\node[main node] (4) [above left of=3] {$d$};

\path[every node/.style={font=\sffamily\small}]
(1) edge node [right] {} (2)
(2) edge node [left] {} (4);

\path[every node/.style={font=\sffamily\small}]
(1) edge node [left] {} (3)
(3) edge node [right] {} (4);
\end{tikzpicture}
\end{center}
As in Example~\ref{ex:two-max-two-min}, we identify each monotone function over $P$ with a tuple $(a,b,c,d)$. In our example the set of monotone functions corresponds to the set of nonnegative integer solutions to the inequalities 
$$
a\le b,\ a\le c,\ b\le d,\ c\le d.
$$
The set of irreducibles can be computed by using Proposition~\ref{prop:irred-pos}, and is 
$$
\{ (0,0,0,1),(0,0,1,1),(0,1,0,1),(0,1,1,1),(1,1,1,1)\}.
$$
These correspond, respectively, to the characteristic functions of the irreducible upper sets 
$$
\{d\}, \{c,d\}, \{b,d\}, \{b,c,d\}, \{a,b,c,d\}.
$$
The only prime element is $(1,1,1,1)$. Notice that $\{a,b,c,d\}$ is the only irreducible upper set fulfilling conditions 1 and 2 in Theorem~\ref{th:char-prime}.

Another way to see that $(1,1,1,1)$ is prime is by observing that it is the only irreducible with the first coordinate not equal to zero. So if it divides an expression, that expression must include $(1,1,1,1)$. This means that $(1,1,1,1)$ is prime.
\end{example}

\begin{proposition}\label{prop:prime1}
Let $f$ be an irreducible monotone function over a finite partially ordered set $P$. Assume that the support of $f$ is not contained in the union of the supports of the other irreducible monotone functions on $P$. Then $f$ is prime.
\end{proposition}

Observe that if $P$ has a minimum (a single minimal element), then $P$ is an irreducible upper set. In this case $P$ fulfills conditions 1 and 2 of Theorem~\ref{th:char-prime}, and consequently $\chi_P$ is a prime element of $\rM(P)$. This is precisely the prime element that appears in Example~\ref{ex:with-ab-a}.

\begin{corollary}\label{cor:minimum}
Let $(P,\le)$ be a finite partially ordered set. If $P$ has a minimum, then $\chi_P$ is a prime element of $\rM(P)$.
\end{corollary}

Indeed, any prime element of $\rM(P)$ comes from an irreducible upper set with a minimum and some extra conditions, as we see next.

\begin{theorem}\label{th:char-prime-comb}
Let $(P,\le)$ be a finite partially ordered set and let $I$ be an irreducible upper set of $P$. Then $\chi_I$ is a prime element of $\rM(P)$ if and only if 
\begin{enumerate}[1.]
\item there is $v\in P$ such that $I=\UAR \{v\}$,
\item if $x\in P\setminus I$ is such that $\UAR \{x\}\cap I\neq \emptyset$, then $x\le v$,
\item the set $\{x\in P\setminus I \mid\ \UAR \{x\}\cap I\neq \emptyset\}$ is either empty or has a maximum.
\end{enumerate}
\end{theorem}

\begin{proof}
\emph{Necessity.} Assume that $I$ is not principal, that is, there exist disjoint nonempty subsets $A,B\subseteq P$ such that $\UAR  (A\cup B)=I$, $\UAR  A\neq I$, $\UAR B\neq I$. Let $J=\UAR A$ and $H=\UAR B$. Then $\chi_I$ divides $\chi_J+\chi_H$, but it does not divide either $\chi_J$ or $\chi_H$. 

If $x\in P\setminus I$ is such that $\UAR \{x\}\cap I\neq \emptyset$, then by condition 2 of Theorem~\ref{th:char-prime}, it follows that $I\subseteq \UAR \{x\}$, and thus $x\le v$.

Now assume that $\{x\in P\setminus I \mid\ \UAR \{x\}\cap I\neq \emptyset\}$ is not empty and that there are at least two maximal elements in this set, say $a$ and $b$. Let $J=\UAR\{a\}$ and $H=\UAR\{b\}$. Both are principal and therefore irreducible upper sets. Then, again by condition 2  of Theorem~\ref{th:char-prime}, it follows that $I\subset J$ and $I\subset H$.  Take $c\in (J\cap H)\setminus I$. Then $a\le c$ and $b\le c$, and so $\UAR\{c\}\cap I$ must be empty, since otherwise neither $a$ nor $b$ would be maximal elements. This implies that $K=(J\cap H)\setminus I$ is an upper set. Then $\chi_{J\cap H}=\chi_K+\chi_I$, and $\chi_I\le_{\rM(P)} \chi_{J\cap H}$, contradicting condition 1 in Theorem~\ref{th:char-prime}.

\emph{Sufficiency.}  If $\{x\in P\setminus I \mid\ \UAR \{x\}\cap I\neq \emptyset\}$ is empty, then $P\setminus I$ is an upper set, and $P$ is the disjoint union of $I$ and $P\setminus I$. Therefore Proposition \ref{prop:prime1} implies that $\chi_I$ is prime.	

Now assume that $\{x\in P\setminus I \mid\ \UAR \{x\}\cap I\neq \emptyset\}$ is nonempty. Assume further that $\chi_I\le_{\rM(P)} f + g$ with $f,g\in \rM(P)$. 	Then $\chi_I+h=f+g$ for some $h\in \rM(P)$. Let $w$ be the maximum of $\{x\in P\setminus I \mid\ \uparrow\!\{x\}\cap I\neq \emptyset\}$. Clearly $\chi_I(w)=0$. Write $\chi_I+h-\chi_{\supp(f)}=f'+g$, with $f'=f-\chi_{\supp(f)}$.

We claim thatt $h-\chi_{\supp(f)}\in \rM(P)$ if $f(w)\neq 0$, and prove it below. Assume for the moment that the claim holds. Then  we can repeat the process until $\chi_I+h'=f'+g'$, with $f'\le_{\rM(P)} f$, $g'\le_{\rM(P)} g$ and $f'(w)=g'(w)=0$. As $\chi_I(v)=1$, either $f'(v)\ge 1$ or $g'(v)\ge 1$ (or both). Assume without loss of generality that $f'(v)\ge 1$. Then $I\subseteq \supp(f')$. If $I=\supp (f')$, then $\chi_I\le_{\rM(P)} f'\le_{\rM(P)} f$, and we are done. So assume that $\supp (f')\setminus I$ is not empty. Take $x$ in this set. Then $x\in \supp (f')$, and as $f'$ is a monotone function on $P$, we have $\uparrow\{x\}\subseteq \supp (f')$. Also $\uparrow\{x\}\cap I$ is empty, because otherwise, $x\le w$, and thus $w\in \supp (f')$, which is impossible. Hence $\uparrow\{x\}\subseteq \supp (f')\setminus I$, and this means that $J=\supp (f')\setminus I$ is an ideal.  Hence $\chi_I+\chi_J=\chi_{\supp (f')}$, and $\chi_I\le_{\rM(P)} \chi_{\supp (f')}\le_{\rM(P)} f$.

It remains to prove the claim above. The first step is that $h-\chi_{\supp(f)}$ has nonnegative values. By assumption $f(w)>0$, and since $w\notin I$, we have $h(w)>0$ as well. If follows that $h(z) >0$ for all $z\in I$ since $I\subset \UAR (w)$. Pick $x\in\supp(f)$. If $x\notin I$, then $h(x)>0$, since $\chi_I+h=f+g$. If $x\in I$, then $h(x)>0$ as well, as just seen.

Now we turn to the monotonicity of $h'=h-\chi_{\supp(f)}$. It is enough to check that $h'(x) \le h'(y)$ if $y$ is a cover of $x$. Since
$$
h' = f' +g -\chi_I,\qquad f'=f-\chi_{\supp(f)}\in \rM(P),\ g\in \rM(P),
$$
we can further assume that $x\notin I$, $y\in I$. But this reduces the question to the case in which $x=w$ and $y=v$. The assumption $f(w)>0$ implies that $\chi_{\supp(f)}(w) = \chi_{\supp(f)}(v)=1$, and we are done. 
\end{proof}

Observe that we can recover Corollary~\ref{cor:minimum} easily with this new characterization.

\section{The monoid of arithmetic multiplicities}
\subsection{Recap on matroids and arithmetic matroids}\label{sec:recalls}

We collect here some basic definitions in order to set some notation. For background on matroid theory we refer, for instance, to Oxley's textbook \cite{Oxley}, while our presentation of arithmetic matroids follows mostly \cite{BM}.
\medskip

\newcommand{\rk}{\operatorname{rk}}

A {\em matroid} is given by a pair $(E,\rk)$, where $E$ is a finite
set and $\rk : 2^E \to \mathbb N$ is a function such that, for all $X,Y\subseteq E$,
\begin{itemize}
\item[(R1)] $\rk(X)\leq \vert X \vert$,
\item[(R2)] $X\subseteq Y$ implies $\rk(X)\leq \rk(Y)$,
\item[(R3)] $\rk(X\cup Y) + \rk(X\cap Y) \leq \rk(X) + \rk (Y)$.
\end{itemize}

Given an element $e\in E$, we can define two matroids on the set $E\setminus \{e\}$: the \emph{deletion} $M_1$ having rank function $\rk_1$ which is simply the restriction of $\rk$, and the \emph{contraction} $M_2$ having rank function $\rk_2$ defined as $\rk_2(A) \doteq \rk(A\cup\{e\}) - \rk(\{e\})$.

A {\em molecule} is a triple $(R,F,T)$ of pairwise disjoint subsets of $E$ such that, for every $A\subseteq E $ with $R\subseteq A \subseteq R\cup F\cup	T$,
$$
\rk(A) = \rk(R) + \vert A\cap F \vert.
$$

As remarked in \cite{BM}, $(R,F,T)$ is a molecule if and only if, after deleting the elements in $E\setminus (R\cup F\cup T)$ and contracting the elements in $R$, $F$ becomes a set of coloops and $T$ becomes a set of loops.
We say that a molecule is \emph{nontrivial} if both $T$ and $F$ are nonempty.

A {\em quasi-arithmetic matroid} is a triple $(E,\rk,m)$ where $(E,\rk)$ is a matroid, and $m: 2^E \to \mathbb N$ is a function satisfying the following axioms.
\begin{itemize}
\item[(A1)] For all $A\subseteq E$ and all $e\in E$,

if $\rk(A\cup \{e\})> \rk(A)$, $m(A)$ divides $m(A\cup \{e\})$;

if $\rk(A\cup \{e\})= \rk(A)$, $m(A\cup \{e\})$ divides $m(A)$.

\item[(A2)]  For every molecule $\alpha=(R,F,T)$ of $(E,\rk)$,
$$
m(R)m(R\cup F\cup T) = m(R\cup F) m(R\cup T).
$$
\end{itemize}
A function $m$ for which these axioms hold is known as a \emph{multiplicity function} on $(E,\rk)$.

An \emph{arithmetic matroid} is a quasi-arithmetic matroid satisfying the following additional axiom:
\begin{itemize}

\item[(P)] For every molecule $\alpha=(R,F,T)$ of $(E,\rk)$,
$$ 
(-1)^{\vert T\vert}
\sum_{R\subseteq A \subseteq R\cup F\cup T}
(-1)^{\vert (R\cup F\cup T) \setminus A \vert} m(A)\geq 0.
$$
\end{itemize}

Axiom $(P)$ was introduced to assure the positivity of an invariant called the arithmetic Tutte polynomial, and is motivated by applications to geometry, while axioms $(A1)$ and $(A2)$ have a more algebraic nature.
Given a matroid $M=(E,\rk)$, we denote by $\mathcal{Q}(M)$ the set of quasi-arithmetic matroids of the form $(E,\rk,m)$, and by $\mathcal{A}(M)$ the set of arithmetic matroids of the form $(E,\rk,m)$.

Given two (possibly different) functions $m', m'' : 2^E \to \mathbb{N}$, let us consider their (point-wise) product $m$, i.e. the function defined as $m(A):= m'(A)m''(A)$ for all $A\subseteq E$. Clearly if  $m', m''$ satisfy axioms (A1) and (A2), also their product $m$ does. So $\mathcal{Q}(M)$ is a commutative monoid, whose unit is the multiplicity identically equal to $1$. Although it is less obvious, also the axiom (P) is preserved by the product, that is, the following result holds.

\begin{theorem}[Delucchi-Moci \cite{DeMo}]
If both $(E,\rk,m')$ and $(E,\rk,m'')$ are arithmetic matroids, then $(E,\rk,m'm'')$ is also an arithmetic matroid. In other words, $\mathcal{A}(M)$ is a submonoid of $\mathcal{Q}(M)$.
\end{theorem}

We will now make the first steps towards understanding the structure of these monoids.

\subsection{Slicing quasi-arithmetic matroids}\label{subs-slicing}

For any  $(E,\rk,m)$ in $\mathcal{Q}(M)$ and every prime $p$, set $v_p(m(A))$ to be the exponent of $p$ in the decomposition of the integer $m(A)$, with $A\subseteq E$. 

The set 
$$
\mathcal{Q}_p(M)=\{v_p\circ m\mid (E,\rk,m)\in \mathcal{Q}(M)\}
$$
is an additive submonoid of $\NN^{2^E}$. Following the approach of \cite[Section 5]{Fink-Moci(JEMS)} one can view $\mathcal Q_p(M)$ as the \emph{localization} of $\mathcal Q(M)$ at the prime $p$. 

\begin{remark}
After this slicing, axiom (A1) consists of inequalities that cut out a polyhedral cone, while axiom (A2) is made of equalities that determine a vector subspace, intersecting the cone into a polytope, whose set of points with nonnegative integer coordinates is  $\mathcal{Q}_p(M)$. It would be interesting to understand the properties of this polytope, and its relation with the polytope described in \cite{Fink-Moci(Advances)}.
\end{remark}

Clearly, for every two primes $p$ and $q$, the monoid $\mathcal{Q}_p(M)$ is isomorphic to $\mathcal{Q}_q(M)$. So we get the following result.

\begin{proposition}\label{infprod}
For every matroid $M$, the monoid  $\mathcal{Q}(M)$ is the direct product of monoids $\mathcal{Q}_p(M)$ (one for each prime number $p$), which are all isomorphic to each other. The projections of every element of $\mathcal{Q}(M)$ on the factors $\mathcal{Q}_p(M)$ are nontrivial only for a finite set of primes.  
\end{proposition}

\begin{remark}\label{unfortunately}
Unfortunately, the same slicing approach does not work in general for the submonoid $\mathcal{A}(M)$. Indeed, an inequality of type (P) is not equivalent to the system of inequalities given by its $p-$slices. 
\end{remark}

\subsection{The digraph associated to a matroid}

Let $M=(E, \rk)$ be a matroid. Define the graph $G_M$ as the (oriented) simple graph with vertices $2^E=\mathcal{P}(E)$. Two subsets are connected by an edge if and only if they differ from each other by adding one element: there is an edge directed from $A$ to $A\cup\{e\}$ if $\rk(A)<\rk(A\cup\{e\})$, and otherwise there is an edge from $A\cup\{e\}$ to $A$.

This graph gathers the inequalities on $v_p\circ m$ derived from (A1).
In this way, we get an acyclic directed graph $G_M$. The sinks of $G_M$ are precisely the bases of the matroid. Moreover, since the graph is acyclic and finite, for any vertex there exists a (non unique) directed path leading to a sink. Given a subset, a possible choice of such a path corresponds to  removing elements until getting an independent set, and then adding elements till  completing this independent set to a basis. By (A1), each of these operations corresponds to an edge oriented in the correct direction.

\begin{lemma}
A subset $B$ of $E$ is a basis for $M$ if and only if it is a sink in $G_M$.
\end{lemma}

\begin{proof}
Bases are sets such that whenever we add a new element the rank does not increase, and if we remove an element, then the rank decreases. This means precisely that the corresponding vertex in $G_M$ is a sink, that is, there are only incoming edges arriving to it.
\end{proof}

\begin{lemma}
The graph $G_M$ is acyclic.
\end{lemma}

\begin{proof}
Assume that $A=A_1,\ldots, A_n=A$ is a cycle starting in $A$. Then $\rk(A_{i})\le \rk(A_{i+1})$ for all possible $i$. Thus $\rk(A_i)=\rk(A)$ for all $i$. But then $A_{i+1}$ must have one element fewer than $A_i$ for all $i$, and this is impossible, because we are starting and ending in $A$.
\end{proof}

\begin{example}
Let $E=\{a,b\}$ and let $\rk(A)=1$ if $a\in A$ and $0$ otherwise. 
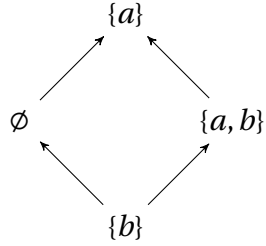
\begin{center}
\begin{tikzpicture}[->,>=stealth',auto,node distance=2cm, main node/.style={}]

\node[main node] (1) {$\{b\}$};
\node[main node] (2) [above left of=1] {$\emptyset$};
\node[main node] (3) [above right of=1] {$\{a,b\}$};
\node[main node] (4) [above left of=3] {$\{a\}$};

\path[every node/.style={font=\sffamily\small}]
(1) edge node [right] {} (2)
(2) edge node [left] {} (4);

\path[every node/.style={font=\sffamily\small}]
(1) edge node [left] {} (3)
(3) edge node [right] {} (4);
\end{tikzpicture}
\end{center}

\end{example}

Observe that in the above example, if $v_p(m(\{a\}))=0$, then (A1) implies that $v_p(m(A))=0$ for all $A$ such that there is a path from $A$ to $\{a\}$ in $G_M$, $M=(E, \rk)$. Thus in this case all nodes have $v_p\circ m$ equal to zero. Also $v_p(m(\emptyset))=0$ or $v_p(m(\{a,b\}))=0$ force $v_p(m(\{b\}))=0$. 

\begin{remark}\label{prop:pre-gm-mf}
Since $G_M$ is acyclic, its reflexive-transitive closure induces a partial order on $2^E$. The monoid of monotone functions over this poset is naturally  isomorphic to the monoid of functions 
$\{v_p\circ m\}$ where $m$ ranges over all the functions $m:2^E\to\mathbb N$ for which  axiom (A1) holds.
\end{remark}

\subsection{Uniform matroids}

From now on we focus on uniform matroids. The \emph{uniform matroid} $\operatorname{U}(r,n)$ is the matroid having $|E|=n$ and for every $A\subseteq E$, $\rk(A)=\min (|A|,r)$. It is realized by $n$ vectors in generic position in an $r-$dimensional vector space. This class of matroids is fundamental in matroid theory; moreover in this case quasi-arithmetic matroids are simpler to study, since an axiom becomes redundant, as the next results show.

\begin{lemma}\label{lem:uniform-no-molecules}
A matroid is uniform if and only if it has no nontrivial molecules.
\end{lemma}

\begin{proof}
\emph{Necessity.} Let $\alpha=(R,F,T)$ be a molecule. It is well-known that the contraction of a uniform matroid by any subset is a uniform matroid. Thus, after contracting a uniform matroid by $R$, we get a uniform matroid that by definition either has no coloops (if its rank is 0) or has no loops (otherwise). So either $F$ or $T$ is empty.

\emph{Sufficiency.} If a matroid of rank $r$ is not uniform, then by definition there exists a subset $A$ of cardinality $r$ that is not a basis. Then we can extract from it a maximal independent subset $I$. So $A= I \cup \{x_1,... x_k\}$. Furthermore $I$ can be completed to a basis, that is, there exist elements $y_i$ such that $I \cup \{y_1,... y_k\}$ is a basis. Then ($R=I$, $F=\{y_1,... y_k\}$, $T=\{x_1,... x_k\}$) is a nontrivial molecule.
\end{proof}

This lemma has a remarkable consequence: for uniform matroids axiom (A2) becomes redundant. Then a multiplicity  function $m$ gives a quasi-arithmetic matroid if and only if it satisfies (A1). 
Therefore we can slice ``prime by prime'' the monoid $\mathcal{Q}(\rU(r,n))$ like in Proposition~\ref{infprod}, and Remark~\ref{prop:pre-gm-mf} yields the following result.

\begin{proposition}\label{prop:gm-mf}
For every prime integer $p$, $\mathcal{Q}_p(\rU(r,n))$ is isomorphic to the monoid of monotone functions on $2^E$ with the order induced by $G_{\rU(r,n)}$.
\end{proposition}

Note that axiom $(P)$ yields nontrivial inequalities also for trivial molecules. Hence, for the reasons explained in Remark \ref{unfortunately}, the slicing approach does not apply to $\mathcal A(\rU(r,n))$. Hence the following problem is left for future work.

\begin{problem}
Describe the structure of the submonoid $\mathcal A(\rU(r,n))$ of $\mathcal Q(\rU(r,n))$.
\end{problem}

\subsection{The Gorenstein property and the Cohen-Macaulay type}

We can now apply the results proved in Section~\ref{sec:monotone-functions} to obtain a description of $\mathcal{Q}_p(\rU(r,n))$.

\begin{theorem}\label{prop:type}
Let $n$ be a positive integer.
\begin{enumerate}[1.]
\item  $\mathcal{Q}_p(\rU(k,n))$ is Gorenstein if and only if $k\in\{0, n/2, n\}$.
\item If $n\ge 2$, then the type of $\mathcal{Q}_p(\rU(n-1,n))$ is $n-1$.
\item If $n\ge 4$,  the type of $\mathcal{Q}_p(\rU(n-2,n))$ is $\sum_{i=1}^{n-3} (n-2-i)^n$.
\end{enumerate}
\end{theorem}

\begin{proof}
By Theorem~\ref{theo:Gor} $\mathcal{Q}_p(\rU(k,n))$ is Gorenstein if and only if there exists a level function, and this is equivalent to the property that all maximal chains that connect minimal and maximal elements have the same length. This property is satisfied exactly in the cases listed in 1.

For 2 we apply Proposition~\ref{rem:type} whose hypotheses are evidently satisfied. In the particular case of $\mathcal{Q}_p(\rU(n-1,n))$, the chains joining the empty set with the sets of $n-1$ elements (the bases) must have values starting with 1 in the empty set, and ending with $n$ in the sets of $n-1$ elements. So the only value left to be assigned is that on the whole set. Since it must be less than $n$, the possibilities are in the set $\{1,\ldots,n-1\}$. For $i\in \{1,\ldots,n-1\}$, define $f_i$ as $f_i(A)=|A|+1$ if $A$ is a proper subset of $\{1,\ldots,n\}$, $|A|<n$, and $f_i(\{1,\ldots,n\})=i$. The functions $f_i$ are exactly those that appear in Proposition~\ref{rem:type}, and we have $n-1$ of them.

For $\rU(n-2,n)$ and $n$ large enough, we have chains joining the empty set with subsets of size $n-2$, all with the same length, and then chains joining the whole set with sets of size $n-1$ and sets of size $n-2$. Thus for the chains joining the empty set with subsets of size $n-2$ the minimal generators of the interior of the cone must have values ranging from $1$ to $n-1$. The value of these maps in the whole set can be between $1$ and $n-3$. Assume that $f$ is one of the generators and that its value in the whole set is $i\in\{1,\ldots,n-3\}$. Then the value in the $n$ different sets of $n-1$ elements must be in $\{i+1,\ldots, n-2\}$. So we have $(n-2-i)^n$ possibilities. This makes $\sum_{i=1}^{n-3} (n-2-i)^n$. Again one must of course argue that these functions generate the interior minimally.
\end{proof}

Since $\mathcal{Q}_p(\rU(k,n))\cong \mathcal{Q}_p(\rU(n-k,n))$, Proposition~\ref{prop:type} gives the types of all $\mathcal{Q}_p(\rU(k,n))$ with $n\le 6$.

\subsection{Irreducibles and primes}

For every $S\subseteq 2^E$, let $\chi_S$ be the indicator function of $S$. A function $m_S:2^E\to \mathbb N$ such that $\chi_S=v_p\circ m_S$ is, for example, the function defined as $m_S(A)=p$ if $A\in S$, $m_S(A)=1$ otherwise.

The results proved in Section~\ref{sec:monotone-functions} allow us to deduce the following facts.

\begin{theorem}\label{prop:uni_prime}\leavevmode
\begin{enumerate}[1.]
\item The irreducible elements in the monoid $\mathcal{Q}_p(\rU(k,n))$ are the elements $\chi_S$, where $S$ ranges over the upper irreducible sets.

\item The monoid $\mathcal{Q}_p(\rU(k,n))$ has no prime elements for $k\neq 0, n$. 

\item The only prime element in $\mathcal{Q}_p(\rU(0,n))$ is the element $\chi_{\UAR [1,n]}$;  the only prime element in $\mathcal{Q}_p(\rU(n,n))$ is the element $\chi_{\UAR\emptyset}$ .
\end{enumerate}
\end{theorem}

\begin{proof}
The first statement follows immediately from Proposition 2. The second and third statements are a consequence of Theorem~\ref{th:char-prime-comb}: indeed this criterion is clearly satisfied by the element $\chi_{\UAR\emptyset}$ if $k=n$ and by the element $\chi_{\UAR[1,n]}$ if $k=0$; moreover it is clearly violated by any other element.
\end{proof}

\begin{remark} 
When the rank of the uniform matroid is 0 (or dually, when it is maximal) then the irreducible monotone functions on the directed graph are simply what are called monotone Boolean functions. Hence the sequence of the number of irreducible multiplicities for the uniform matroid $\operatorname{U}(0,n)$ is given by the so-called \emph{Dedekind numbers}: see \cite{dedekind} or the OEIS \cite{oeis} sequence A014466: 

$$
\{1, 2, 5, 19, 167, 7580, 7828353, 2414682040997, 56130437228687557907787,\dots\}.
$$
(The following terms of the sequence are unknown.)

However, the number of irreducibles for $\operatorname{U}(k,n)$, $k\not\in\{ 0, n\}$ seem not correspond to any known OEIS (bi)sequence. For $n\le 6$ the numbers are given in Table~\ref{NumberIrred} (up to the symmetry between  $\operatorname{U}(k,n)$ and $\operatorname{U}(n-k,n)$). These numbers where computed with \texttt{Normaliz} \cite{normaliz} and a special program for $n=6$.
\end{remark} 

Despite many efforts, a closed formula for Dedekind numbers has never been found. Then it seems hopeless that a closed formula could be found for the number of irreducibles of $\mathcal{Q}_p(\operatorname{U}(k,n))$. However, it would be interesting to give some estimates or bounds.

\begin{problem}
Provide upper and lower bounds for the number of irreducibles of $\mathcal{Q}_p(\operatorname{U}(k,n))$.
\end{problem}

\begin{table}
\begin{tabular}{r|r|r|r|r|}
&$k=0$&1&2&3\\
\hline
$n=1$&2&&&\\
\hline
2&5&5&&\\
\hline
3&19&20&&\\
\hline
4 &167&228&290&\\
\hline
5&7580&13727&47507&\\
\hline
6&7828353&15568259&242938059&1604376245\\
\hline
\end{tabular}
\vspace*{2ex}
\caption{Number of irreducibles for $\operatorname{U}(k,n)$.}\label{NumberIrred}
\end{table}

The following two conjectures were checked for all the matroids on $n\leq 4$ elements. Beyond experimental evidence, the intuition is that matroids having more bases are likely to give rise to monoids with more irreducibles.

\begin{conjecture}
The number of irreducibles of $\mathcal{Q}_p(\operatorname{U}(\lfloor n/2\rfloor,n))$ is an upper bound for the number of irreducibles of any uniform matroid on $n$ elements.
\end{conjecture}

\begin{conjecture}
The number of irreducibles of $\mathcal{Q}_p(\operatorname{U}(k,n))$ is an upper bound for the number of irreducibles of $\mathcal{Q}_p(M)$, where $M$ ranges over all matroids of rank $k$ on $n$ elements.
\end{conjecture}

\end{document}